\documentclass{amsart}
\usepackage{amsmath}
\usepackage{framed,comment,enumerate}
\usepackage[pdftex]{graphicx}
\usepackage{epstopdf}
\usepackage{xcolor, framed}
\usepackage{color}

\def\R {\mathbb{R}}
\def\N{\mathbb{N}}
\def\eps{\varepsilon}

\def\OP{obstacle problem}
\def\FB{free boundary }

\def\Sph{\mathbb{S}^{d-1}}
\def\HPP{\{x_d=0\}}

\def\SphLap{\Delta_{\Sph}}
\def\SphGrad{\nabla_{\Sph}}

\def\PosS{\{u>0\}}
\def\ConS{\{u=0\}}

\def\PE{\mathcal{P}_{2k}}
\def\PO{\mathcal{P}_{2k+1}}

\def\pbar{\overline{p}}

\def\hu{\hat{u}}

\def\ddd{\frac{\partial}{\partial x_d}}

\def\HausSurf{d\mathcal{H}^{d-1}|_{S^{\pm}_\eta}}

\newtheorem{thm}{Theorem}[section]

\newtheorem{lem}{Lemma}[section]
\theoremstyle{definition}
\newtheorem{defi}{Definition}[section]
\newtheorem{rem}{Remark}[section]
\numberwithin{equation}{section}

\title[Points of integer frequencies]{Contact points with integer frequencies in the thin obstacle problem } 

\author{Ovidiu Savin}
\address{Department of Mathematics,	Columbia University, New York, USA}
\email{savin@math.columbia.edu}

\author{Hui Yu}
\address{Department of Mathematics,	Columbia University, New York, USA}
\email{ huiyu@math.columbia.edu}
\thanks{O.~S.~is supported by  NSF grant DMS-1800645.}
\thanks{H.~Y.~is supported by NSF grant DMS-1954363.}

\begin{document}

\begin{abstract}
For the thin \OP, we develop a unified approach that leads to rates of convergence to blow-up profiles at contact points with integer frequencies. For these points, we also obtain a stratification result.\end{abstract}

\maketitle
%%%%%%%%%%%%%%%%%%%%%%%%%%%%%%%%%%%%%%%%%%%%%%%%%%
\section{Introduction}
The \textit{thin \OP} studies  the following system 
\begin{equation}\label{TOP}
\begin{cases}
\Delta u\le 0 &\text{ in $B_1$,}\\
u\ge 0 &\text{ on $B_1\cap\HPP$,}\\
\Delta u=0 &\text{ in $B_1\cap(\PosS\cup\{x_d\neq 0\})$.}
\end{cases}
\end{equation}  
Here we denote by $B_1$ the unit ball in the  Euclidean space $\R^d$. For a point $x\in\R^d$, we decompose its coordinate as $x=(x',x_d)$ with $x'\in\R^{d-1}$ and $x_d\in\R$. Since the odd part of the solution, $(u(x',x_d)-u(x',-x_d))/2$, is harmonic, it is customary to remove it and assume that \textit{the solution $u$ is even in the $x_d$-direction}. 

After earlier results by Richardson \cite{R} and Uraltseva \cite{U}, Athanasopoulos and Caffarelli showed in \cite{AC} that the solution is locally Lipschitz in $B_1$, and is locally $C^{1,1/2}$ when restricted to either $B_1\cap\{x_d\ge 0\}$ or $B_1\cap\{x_d\le 0\}$. This  optimal regularity of the solution opened the door to  the study of the \textit{contact set} 
$$\Lambda(u):=\ConS\cap\HPP$$ and the \textit{free boundary} $$\Gamma(u):=\partial\PosS\cap\HPP.$$

In Athanasopoulos-Caffarelli-Salsa \cite{ACS}, the authors made a breakthrough by applying Almgren's monotonicity formula  to show that for each point $q\in\Lambda(u)$, there is a constant $\lambda_q$ such that  
$$\|u\|_{\mathcal{L}^2(\partial B_r(q))}\sim r^{\frac{d-1}{2}+\lambda_q}$$
 as $r\to0.$ This constant $\lambda_q$ is called the \textit{frequency of the solution at $q$}. They also showed that the normalized rescalings converge to a \textit{blow-up profile}, that is, 
 \begin{equation}\label{FirstConvergence}u_{q,r}(\cdot):=r^{\frac{d-1}{2}}\frac{u(r\cdot+q)}{\|u\|_{\mathcal{L}^2(\partial B_r(q))}}\to u_0\end{equation} 
 along a subsequence of $r\to 0.$  The limit $u_0$ is a $\lambda_q$-homogeneous solution to \eqref{TOP} in $\R^d.$

It is interesting to  study admissible values of frequencies, to classify homogeneous solutions, and to establish regularity of the contact set/ the free boundary. So far this program is  completed only when $d=2.$  See, for instance, Petrosyan-Shahgholian-Uraltseva \cite{PSU}.

Let $\mathcal{A}$ denote the set of admissible frequencies, that is, $$\mathcal{A}=\{\lambda\in\R: \text{ there is a non-trivial $\lambda$-homogeneous solution to $\eqref{TOP}$}\}.$$ For a solution $u$ to \eqref{TOP} and $\lambda\in\mathcal{A}$,  let $\Lambda_\lambda(u)$ denote the set of contact points with frequency $\lambda$, that is,
$$\Lambda_\lambda(u)=\{q\in\Lambda(u):\lambda_q=\lambda\}.$$

In general dimensions, Athanasopoulos-Caffarelli-Salsa \cite{ACS} showed that $\mathcal{A}\subset \{1,\frac 32\}\cup [2,+\infty).$ Explicit examples give that $\mathbb{N}\cup \{2k-\frac 12:k\in\N\}\subset\mathcal{A}.$ See, for instance, \cite{PSU}. Around each $m\in\N$, there is a frequency gap \cite{CSV, SY2}, in the sense that we can find $\alpha_m>0$, depending on $m$ and $d$, such that 
$$\mathcal{A}\cap(m-\alpha_m,m+\alpha_m)=\{m\} \text{ for each $m\in\N$.}$$

As for the classification of homogeneous solutions and for the regularity of the contact set/ the free boundary, most results center on  points with frequencies in $\{\frac 32\}\cup\N.$ 

Already in \cite{ACS}, it was known that the only $\frac 32$-homogeneous solutions are rotations and multiples of 
$$u_{\frac 32}(x',x_d)=Re(x_{d-1}+i|x_d|)^{3/2}.$$ 
For a solution to \eqref{TOP}, the set $\Lambda_{\frac 32}(u)$  is relatively open in $\Gamma(u)$. The free boundary is an analytic manifold of dimension $(d-2)$ near $\Lambda_{\frac 32}(u)$ \cite{DS,KPS}.

For an even integer $2k$, Garofalo-Petrosyan \cite{GP} classified all $2k$-homogeneous solutions to \eqref{TOP} as 
\begin{align}\label{EvenHomSolution}
\PE^+=\{p: \text{ }&\Delta p=0 \text{ and } x\cdot\nabla p=2kp \text{ in $\R^d$,}   \\& p(\cdot,0)\ge 0 \text{ and } p(\cdot,x_d)=p(\cdot,-x_d)\}.\nonumber
\end{align} 
Let $u$ be a solution to \eqref{TOP} and $q\in\Lambda(u).$ Garofalo-Petrosyan gave a geometric characterization of contact points with even frequencies  
$$q\in\Lambda_{2k}(u) \text{ for some $k\in\N$} \iff \mathcal{H}^{d-1}(\Lambda(u)\cap B_r(q))=o(r^{d-1}) \text{ as $r\to 0.$}$$ 
Here $\mathcal{H}^{d-1}$ is the $(d-1)$-dimensional Hausdorff measure.  In particular, we have $\Lambda_{2k}(u)\subset\Gamma(u).$ They also showed that $\Lambda_{2k}(u)$ is locally covered by $C^{1}$-manifolds. By quantifying the rate of convergence in \eqref{FirstConvergence} at points in $\Lambda_{2k}(u)$, the regularity of the covering manifolds was improved to $C^{1,\log^c}$ in Colombo-Spolaor-Velichkov \cite{CSV}.

More recently, there has been some interest in the study of contact points with odd frequencies, mainly motivated by the connection to the singular set in the obstacle problem, see \cite{FRS}. While points in $\Lambda_1(u)$ lie in the interior of $\Lambda(u),$ it remains an open question whether $\Lambda_{2k+1}(u)$ can contain points in the free boundary for $k\ge 1$. On the other hand, it is known that around $\Lambda_{2k+1}(u)$, the contact set has full density, that is,
  $$q\in\Lambda_{2k+1}(u) \text{ for some $k\in\N$} \iff \mathcal{H}^{d-1}(\{u>0,x_d=0\}\cap B_r(q))=o(r^{d-1})\text{ as $r\to 0.$}$$
  See, for instance, Proposition 7.1 in  Fern\'andez-Real \cite{Fe}.
  
 The family of $(2k+1)$-homogeneous solutions to \eqref{TOP} was recently classified by Figalli, Ros-Oton and Serra in \cite{FRS} as
\begin{align}\label{OddHomSolution}
\PO^+=\{p:    \text{ }&\Delta p=0 \text{ in $\{x_d\neq0\}$ and } \Delta p\le 0 \text{ in $\R^d$,} \\ &x\cdot\nabla p=(2k+1)p \text{ in $\R^d$,}\nonumber \\&p(\cdot,0)= 0 \text{ and } p(\cdot,x_d)=p(\cdot,-x_d)\}.\nonumber
\end{align} They also proved  uniqueness of the blow-up profile $u_0$ in \eqref{FirstConvergence} at $q\in\Lambda_{2k+1}(u)$.

Along a different direction, Focardi-Spadaro  proved that the free boundary $\Gamma(u)$ is countably $(d-2)$-rectifiable in \cite{FoS1, FoS2}. They also showed  that outside a set of dimension at most $(d-3)$, all free boundary points have frequencies in $\{2k,2k+1,2k-\frac 12: k\in\N\backslash\{0\}\}$. For generic boundary data, Fern\'andez-Real and Ros-Oton showed that the free boundary is smooth outside a set of dimension at most $(d-3)$ in \cite{FeR}.

In this paper, we focus on contact points with integer frequencies, that is, points in $\cup_{k\in\N}\Lambda_k(u).$ Around these points, we develop a unified approach that gives a uniform rate for the convergence  in \eqref{FirstConvergence}.

Our main result is:
\begin{thm}\label{MainResult}
Suppose that $u$ is a solution to the thin obstacle problem \eqref{TOP}, and that $0\in\Lambda_m(u)$ for some $m\in\N$.

If $m=2k+1$ is odd, then there is a constant $\alpha\in(0,1)$, depending only on $k$ and the dimension $d$, such that 
\begin{equation}\label{MainResultOdd}
u(x)=p(x)+O(|x|^{2k+1+\alpha}) \text{ as $x\to0$}
\end{equation}  for some $p\in\PO^+.$

If $m=2k$ is even, then there is a constant $c>0$, depending only on $k$ and $d$, such that 
\begin{equation}\label{MainResultEven}
u(x)=p(x)+O(|x|^{2k}(-\log|x|)^{-c}) \text{ as $x\to0$}
\end{equation}  for some $p\in\PE^+.$
\end{thm} 

\begin{rem} While it is known in \cite{FRS} that  rescaled solutions converge to some $p\in\PO^+$ at points with odd frequencies, this is the first time a quantified rate of convergence has been obtained.  Corresponding results at even frequency points were known in Colombo-Spolaor-Velichkov \cite{CSV}. Our method is different and applies to all points with integer frequencies. It also leads to an improved exponent $c$ in \eqref{MainResultEven}, and in the corresponding $\log$-epiperimatric inequality from \cite{CSV} at points with even frequencies, see Remark \ref{ImprovedEpi}.
\end{rem}

With a standard application of Whitney's extension theorem and the implicit function theorem, Theorem \ref{MainResult} leads to the following stratification result of $\Lambda_m(u)$:
\begin{thm}\label{MainStratification}
Suppose that $u$ is a solution to the thin obstacle problem \eqref{TOP}.

For each $m\in\N$, we have the following decomposition $$\Lambda_m(u)=\cup_{j=0,1,\dots, d-2}\Lambda_m^j(u).$$ 

The lowest stratum $\Lambda_m^0(u)$ is locally isolated. 

If $m$ is odd, then $\Lambda_m^j(u)$ is locally covered by a $j$-dimensional $C^{1,\alpha}$ manifold for each $j=1,\dots,d-2.$

If $m$ is even, then $\Lambda_m^j(u)$ is locally covered by a $j$-dimensional $C^{1,\log}$ manifold for each $j=1,\dots,d-2.$
\end{thm} 

\begin{rem}
Points in $\Lambda_{1}(u)$ and $\Lambda_3^0(u)$ lie in the interior of the contact set $\Lambda(u)$. It remains open whether other strata of $\Lambda_{2k+1}(u)$ can contain points on the free boundary,  see Remark \ref{Interior}.
\end{rem}

To obtain the results at even-frequency points, the approach taken by Colombo-Spolaor-Velichkov  \cite{CSV} is based on the decomposition of the energy in terms of Fourier modes. This leads to a $\log$-epiperimetric inequality for the $2k$-Weiss energy functional (see \eqref{Weiss}). On the other hand, our method is based on the classic technique of linearization as in De Silva \cite{D}.  By  working directly in the physical space instead of the Fourier space, it seems that we are able to get more detailed information. 

The main challenge is that solutions to the linearized problem do not have to satisfy the constraints in \eqref{TOP}  (they might fail $\Delta u\le 0$ or 
$u(x',0)\ge 0$). In our approach, this issue is fixed by solving a `boundary layer problem' near the hyperplane $\HPP.$  For each unconstrained $m$-homogenous harmonic polynomial $p$, we associate its approximation $\bar p$ that satisfies the constraints on $\{x_d=0\}$ and is harmonic up to an error $\kappa_p$ away from this hyperplane. We use the class of functions $\bar p$ to approximate the solution $u$ inductively in dyadic balls $B_r$, while keeping track of the rescaled error $\eps \ge \kappa_p$. We introduce the notation $u \in \mathcal{S}_{m}(p,\eps,r)$ when a solution $u$ is $\eps$-approximated by $\bar p$ at scale $r$, see Definition \ref{DefWellApprox}.

With this notation, the main lemma is the following: 
\begin{lem}\label{Dichotomy}
Given  $m\in\N$, there are constants, $\tilde{\eps}$, $r_0$, $c$ small, $C$ big, such that 

If $u\in\mathcal{S}_{m}(p,\eps,1)$ with $\eps<\tilde{\eps}$ and $1\le \|p\|_{\mathcal{L}^2(\Sph)}\le 2$, then we have the following dichotomy: 

a) Either $$W_{m}(u;1)-W_{m}(u; r_0)\ge c\eps^2,$$ and $$u\in\mathcal{S}_m(p,C\eps,r_0);$$

b) or $$u\in\mathcal{S}_{m}(p',\frac12\eps,r_0)$$ for some $p'$ with $$\|p'-p\|_{\mathcal{L}^2(\Sph)}\le C\eps.$$
\end{lem} 

Here the integer $m$ denotes the frequency of the contact point, and $W_m$ is the Weiss energy functional, see \eqref{Weiss}. 

Lemma \ref{Dichotomy} states that when moving to a smaller scale, we can improve, by a definite amount, either the Weiss energy or  the error in approximation. On the other hand, the Weiss energy is controlled by the error $\eps$ in the approximation, see Lemma \ref{WeissComparison2}. It follows that both quantities decay in a quantified fashion. 
Similar ideas have been applied to the nonlinear obstacle problem in \cite{SY1} as well as the triple-membrane problem in \cite{SY3}.

This paper is organized as follows: In Section 2, we collect some preliminary results and introduce the boundary layer problem. In Sections 3 and 4, we prove Lemma \ref{Dichotomy} at points with the odd and even frequencies. These two sections contain the heart of this work. In Section 5, we conclude with the proof of our main results.

%%%%%%%%%%%%%%%%%%%%%%%%%%%%%%%%%%%%%%%%%%%%%%%%%%%%%%%%%%%%%%%%%%%%%%%%%%%%%%%%%%%%%%%%%%%%%%%%%%%%%%%%%%%%%%%%%%%%%%%%%%%%%%%%%%%%%%%%%%%%%%%%%%%%%%%%%
\section{Preliminaries}
In this section, we collect some useful results and introduce the boundary layer problem. 

Throughout this paper, we denote by $u$ a solution to the thin obstacle problem  \eqref{TOP} on some domain inside $\R^d$ with $d\ge 3.$ This space is decomposed as $$\R^d=\{(x',x_d):x'\in\R^{d-1}, x_d\in\R\}.$$ For a set $E\subset\R^d$, we define the following subsets  relative  to $\HPP$:
$$E'=E\cap\HPP, \text{ } E^+=E\cap\{x_d>0\} \text{ and } E^-=E\cap\{x_d< 0\}.$$In particular, the contact set is $$\Lambda(u)=\{u=0\}'.$$

By applying Almgren's monotonicity formula to the thin \OP, Athanasopoulos, Caffarelli and Salsa showed in \cite{ACS} that the contact set $\Lambda(u)$ can be decomposed according to the frequencies of the contact points. In this work, we focus on points with integer frequencies. Thanks to \cite{GP} and \cite{FRS}, these points can be characterized as 
\begin{equation}\label{EvenContactDef}
\Lambda_{2k}(u)=\{q\in\Lambda(u): u_{q,r}\to u_0\in\PE^+ \text{ as $r\to 0$}\}
\end{equation} and 
\begin{equation}\label{OddContactDef}
\Lambda_{2k+1}(u)=\{q\in\Lambda(u): u_{q,r}\to u_0\in\PO^+ \text{ as $r\to0$}\}.
\end{equation} 
Recall the definition of normalized rescalings $u_{q,r}$ from \eqref{FirstConvergence}. The spaces of homogeneous solutions, $\PE^+$ and $\PO^+$, were introduced in \eqref{EvenHomSolution} and \eqref{OddHomSolution}.

When focusing on a particular frequency, constants depending only on that frequency and the dimension $d$ are called \textit{universal constants}.

%%%%%%%%%%%%%%%%%%%%%%%%%%%%%%%%%%%%%%%%%%%%%%%%%%%
\subsection{Weiss monotonicity formula and consequences}
First used by Weiss for the obstacle problem in \cite{W}, the Weiss monotonicity formula has been indispensable  in the study of free boundary problems. Garofalo-Petrosyan \cite{GP} introduced its analogue to the thin obstacle problem.

For each $\lambda\in\R$, the \textit{$\lambda$-Weiss energy functional} is 
\begin{equation}\label{Weiss}
W_\lambda(u;r)=\frac{1}{r^{d-2+2\lambda}}\int_{B_r}|\nabla u|^2-\frac{\lambda}{r^{d-1+2\lambda}}\int_{\partial B_r}u^2.
\end{equation} 
We collect some of its properties in the following lemma. For its proof, see Theorem 1.4.1 and Theorem 1.5.4 in \cite{GP}.

\begin{lem}\label{PropertyWeiss}
Suppose that $u$ solves the thin obstacle problem in $B_1$. Then for $r\in(0,1)$, we have
\begin{equation}\label{DerOfWeiss}
\frac{d}{dr}W_\lambda(u;r)=\frac{2}{r}\int_{\partial B_1}(\nabla u_r\cdot\nu-\lambda u_r)^2,
\end{equation} where $u_r(x)=u(rx)/r^\lambda.$ In particular, $r\mapsto W_\lambda(u;r)$ is non-decreasing. 

If we further assume that $0\in\Lambda_\lambda(u)$, then $\lim_{r\to 0}W_\lambda(u;r)=0.$
\end{lem} 

Under the same assumptions as in Lemma \ref{PropertyWeiss}, we can integrate \eqref{DerOfWeiss} and apply H\"older's inequality to get
\begin{equation}\label{ChangeInRadial}
\int_{\partial B_1}|u_r-u_s|\le (\log(r/s))^{\frac 12}[W_\lambda(u;r)-W_\lambda(u;s)]^{\frac 12}
\end{equation} for $0<s<r<1.$

%%%%%%%%%%%%%%%%%%%%%%%%%%%%%%%%%%%%%%%%%%%%%%%%%%%
\subsection{The boundary layer problem}\label{TBLP}
When dealing with the linearized problem, we need to work with  polynomials that may fail the constraints in \eqref{TOP}. These polynomials form the following spaces:
\begin{equation}\label{EvenHomPolyn}
\PE=\{p: \text{ }\Delta p=0 \text{ and } x\cdot\nabla p=2kp \text{ in $\R^d$,}   \text{ and } p(\cdot,x_d)=p(\cdot,-x_d)\}
\end{equation} and
\begin{align}\label{OddHomPolyn}
\PO=\{p:    \text{ }&\Delta p=0 \text{ in $\{x_d\neq0\}$, } x\cdot\nabla p=(2k+1)p \text{ in $\R^d$,} 
\\&p(\cdot,0)= 0 \text{ and } p(\cdot,x_d)=p(\cdot,-x_d)\}.\nonumber
\end{align}

Compared with \eqref{EvenHomSolution} and \eqref{OddHomSolution}, polynomials in $\PE$ may fail to be non-negative along $\HPP$, and polynomials in $\PO$ may fail to be superharmonic. We `correct' such error by solving a thin obstacle problem in a boundary layer on the sphere $\Sph$. 

To be precise, for small $\eta>0$, the \textit{boundary layer of width $\eta$} is defined as 
$$L_\eta=\{(x',x_d):|x_d|<\eta|x|\}.$$This is the region trapped by the following surfaces 
$$S^+_\eta=\{(x',x_d):x_d=\eta|x|\} \text{ and } S^-_\eta=\{(x',x_d):x_d=-\eta|x|\}.$$
When there is no ambiguity, we denote their intersections with $\Sph$ by the same expressions.

\begin{rem}Given $m\in\N$, we fix $\eta$ small enough, depending only on $m$ and $d$, so that  the first Dirichlet eigenvalue of the  the operator $\Delta_{\Sph}+\lambda(m)$ in $L_\eta$ is negative, where $\Delta_{\Sph}$ is the spherical Laplacian, and $$\lambda(m)=m(m+d-2).$$ 
\end{rem}

In particular, the following is well-defined:
\begin{defi}\label{ReplacementOfp}
Given $m\in\N$ and $p\in\mathcal{P}_m$, the \textit{replacement of $p$}, denoted by $\overline{p}$, is the minimizer of the following energy
$$w\mapsto \int_{\Sph}|\nabla_{\Sph} w|^2-\lambda(m)w^2$$over functions satisfying $w\ge 0$ on $\HPP$ and $w=p$ outside $L_\eta$.  

Here $\nabla_{\Sph}$ denotes the tangential gradient on $\Sph$.

Denote the difference of $p$ and its replacement $\pbar$ by  $v_p$, that is,
$$v_p=\overline{p}-p.$$
\end{defi}

\begin{rem}\label{LongRemarkForReplacement}
The replacement solves the thin obstacle problem for the operator $\Delta_{\Sph}+\lambda(m)$ in the boundary layer, namely, 
$$\begin{cases}
(\Delta_{\Sph}+\lambda(m))\overline{p}\le 0 &\text{in $L_\eta$,}\\
\overline{p}\ge0 &\text{on $(\Sph)',$}\\
(\Delta_{\Sph}+\lambda(m))\overline{p}=0 &\text{in $L_\eta\cap(\{x_d\neq 0\}\cup\{\overline{p}>0\}).$}
\end{cases}$$
As a result,  we can view $(\Delta_{\Sph}+\lambda(m))\overline{p}$ as a signed measure, supported along $S^{\pm}_\eta$ and $\Sph\cap\{\pbar=0\}'$, of the following form
$$(\Delta_{\Sph}+\lambda(m))\overline{p}=f_pd\mathcal{H}^{d-2}|_{S^{\pm}_\eta}+g_pd\mathcal{H}^{d-2}|_{(\Sph)'}.$$
With an abuse of notation, we denote the $m$-homogeneous extension of $\overline{p}$ and the corresponding extensions of $f_p$ and $g_p$ by the same notations. This way, we have 
$$\Delta\overline{p}=f_pd\mathcal{H}^{d-1}|_{S^{\pm}_\eta}+g_pd\mathcal{H}^{d-1}|_{\HPP}.$$
For each $p\in\mathcal{P}_m$, the following constant, $\kappa_p$, measures the extent to which $p$ fails to be a solution to the thin \OP:
\begin{equation*}\label{Kappa}
\kappa_p:=\int_{S^+_\eta\cap\Sph}f_pd\mathcal{H}^{d-2}.
\end{equation*} 

For these functions and constants, we often omit the subscript $p$ when there is no ambiguity. 
\end{rem}

\begin{lem}\label{Lem22}
Using the notations in Definition \ref{ReplacementOfp} and Remark \ref{LongRemarkForReplacement}, we have

1) $v_p\ge 0$ on $\Sph$, $\kappa_p\ge 0.$

2) There are  universal constants, $c$ and $C$, such that 
$$c\kappa_p\le f_p\le C\kappa_p\text{ on $S^{\pm}_\eta\cap\Sph$}.$$ \end{lem} 

\begin{proof}
The first statement follows directly from the maximum principle. 

To see the second statement, we first note that $v$ is a non-negative harmonic function in $L_\eta^+.$ By the strong maximum principle, it suffices to consider the case when $v>0$ in $L_\eta^+.$

In this case, we can apply the Harnack principle to $v$ inside $(B_2\backslash B_{1/2})\cap L_\eta^+$ to get 
$c\le\frac{\sup_K v}{\inf_K v}\le C,$ where $K=\Sph\cap\{x_d=\frac{1}{2}\eta\}.$

If we denote by $\nu$ the unit normal along  $S_\eta^+$ that is exterior to $L_\eta$, by the boundary Harnack principle, we have 
\begin{equation}\label{22equation}c\inf_K v\le \partial_\nu v(x)\le C\sup_K v \quad \forall x\in\Sph\cap S_\eta^+.\end{equation}

Note that $f=\partial_\nu v$ along $S_\eta^+$, the conclusion follows. 
\end{proof} 

We also have the following bounds for $\kappa_p$:
\begin{lem}\label{BoundsForKappa}
Using the notations in Definition \ref{ReplacementOfp} and  Remark \ref{LongRemarkForReplacement}, and further assume $\|p\|_{\mathcal{L}^2(\Sph)}\le 1$,  we can find a universal constant $C$  such that 
$$\kappa_p\le C\sup_{\Sph} v_p.$$

Moreover, when $m$ is even, we have  $$(\sup_{\Sph}v_p)^{\frac{d-1}{2}}\le C\kappa_p.$$
\end{lem} 

\begin{proof}
Directly  from \eqref{22equation}, we have the bound $\kappa_p\le C\sup_{\Sph} v_p.$

For the second comparison, we note that when $m$ is even, $v=\pbar-p$ solves the thin obstacle problem in $L_\eta$ with $-p$ as the obstacle and $0$ as boundary data. Suppose $$\eps=-p(e_1)=\sup_{(\Sph)'}(-p),$$ where $e_1$ is the unit vector in the $x_1$-direction, then $\sup_{\Sph} v\le\eps$.   Regularity of $p$ gives  $v\ge -p\ge\frac 78 \eps$ in $B'_{c\eps^{1/2}}(e_1)\cap\Sph$, which leads to $v\ge \frac 78 \eps$ in $B_{c\eps^{1/2}}(e_1)\cap\Sph$ by a scaling argument. 

From here we have $v(e_1,\frac{1}{2}\eta)\ge c\eps^{\frac{d-1}{2}}$ by comparing with a truncation and rescaling of the Green's function for $\SphLap$ with a pole at $e_1$ and vanishes outside $B_\eta(e_1)\cap\Sph.$ Combining this with \eqref{22equation} gives the desired result.
\end{proof} 

With this, we have the following control over the size of $\|v\|_{H^1}$ in terms of $\kappa$:

\begin{lem}\label{H1Forv}
Using the notations in Definition \ref{ReplacementOfp} and  Remark \ref{LongRemarkForReplacement}, and further assume $\|p\|_{\mathcal{L}^2(\Sph)}\le 1$, we have a universal constant $C$ such that 
$$\|v_p\|_{H^1(B_1)}\le C\kappa_p^{\frac{1}{2}+\frac{1}{d-1}} \text{ if $p\in \PE$}$$ and 
$$\|v_p\|_{H^1(B_1)}\le C\kappa_p^{1/2} \text{ if $p\in \PO$}.$$
\end{lem} 

\begin{proof} 
We first deal with the case when $p\in\PE$.

With the homogeneity and harmonicity of $p$, we have 
\begin{align*}
\int_{\Sph}|\SphGrad\pbar|^2-\lambda(2k)\pbar^2&=-\int_{\Sph}v(\SphLap+\lambda(2k))\pbar\\
&\le-\sup_{\Sph}v\int_{(\Sph)'}g.\nonumber
\end{align*}

Now let $P$ denote the $2k$-homogeneous harmonic polynomial with $P=1$ on $(\Sph)'$. Then 
$ \int_{\Sph}P(\SphLap+\lambda(2k))\pbar=\int_{\Sph}\pbar(\SphLap+\lambda(2k))P=0 $ gives 
$$-\int_{(\Sph)'}g\sim \int_{S^{\pm}_\eta}f\sim\kappa.$$ Thus 
\begin{equation}\label{EvenPbarEnergy}
\int_{\Sph}|\SphGrad\pbar|^2-\lambda\pbar^2\le C\kappa^{1+\frac{2}{d-1}}
\end{equation} 
by Lemma \ref{BoundsForKappa}. 

Using again the homogeneity and harmonicity of $p\in\PE$, this implies 
$$\int_{\Sph}|\SphGrad v|^2-\lambda v^2=\int_{\Sph}|\SphGrad\pbar|^2-\lambda\pbar^2\le C\kappa^{1+\frac{2}{d-1}}.$$
To conclude the estimate for $p\in\PE$, we simply note that the left-hand side is comparable to $\|v\|_{H_1(B_1)}^2$ when $\eta$ is chosen small.

Now we deal with the case when $p\in\PO$. 

In this case,  note that $p$ is admissible in the minimization problem  in Definition \ref{ReplacementOfp}, we have 
\begin{equation}\label{OddPbarEnergy}\int_{\Sph}|\nabla_{\Sph}\pbar|^2-\lambda(2k+1)\pbar^2\le\int_{\Sph}|\nabla_{\Sph}p|^2-\lambda(2k+1)p^2=0,
\end{equation}
which implies
$$\int_{\Sph}|\nabla_{\Sph}(\pbar-p)|^2-\lambda(\pbar-p)^2\le 2\int_{\Sph}p(\Delta_{\Sph}+\lambda)(\pbar-p)=2\int_{\Sph\cap S_{\eta}^{\pm}}pfd\mathcal{H}^{d-2}.
$$For the last equality, we used $p=0$ along $\HPP$, and $\Delta p=0$ away from $\HPP.$
This gives 
$$\int_{\Sph}|\nabla_{\Sph}v|^2-\lambda v^2\le C\kappa.$$We conclude by noting the left-hand side is comparable to $\|v\|_{H^1(B_1)}^2$ when $\eta$ is small.  \end{proof}

For $p\in\mathcal{P}_m$ with $\pbar\neq p$, we sometimes need to absorb the right-hand side $f_p$ from  Remark \ref{LongRemarkForReplacement}. To do so, we need some auxiliary functions. 

Firstly, let $\varphi_p:\Sph\to\R$ denote the projection of the normalized $f_p$ onto $\mathcal{P}_m$, that is, 
\begin{equation}\label{Defphip}
\varphi_p=\sum \langle \frac{f_p}{\kappa_p},p_j\rangle p_j,
\end{equation} where $\{p_j\}$ is an orthonormal basis for $\mathcal{P}_m$ in $\mathcal{L}^2(\Sph)$ and $$\langle\frac{f_p}{\kappa_p},p_j\rangle=\frac{1}{\kappa_p}\int_{\Sph\cap S^{\pm}_\eta}p_jf_pd\mathcal{H}^{d-2}.$$
 
 In particular, the difference $\frac{f_p}{\kappa_p}-\varphi_p$ is perpendicular to $\mathcal{P}_m$. Consequently, the theory of Fredholm implies that we can find a unique function $H_p$ on $\Sph$, that is even with respect to $x_d$ and satisfyies:
 
 If $m=2k$, then 
  \begin{equation}\label{DefHpEven}
(\SphLap+\lambda(2k))H_p=\frac{f_p}{\kappa_p}-\varphi_p \text{ on $\Sph$};
\end{equation} 

If $m=2k+1$, then 
  \begin{equation}\label{DefHpOdd}
(\SphLap+\lambda(2k+1))H_p=\frac{f_p}{\kappa_p}-\varphi_p \text{ on $(\Sph)^{\pm}$, and } H_p(\cdot,0)=0.
\end{equation} 

If we denote its $m$-homogeneous extension  also by $H_p$, then 
\begin{equation}\label{DefPhip}
\Phi_p:=H_p+\frac{1}{d+2m-2}\varphi_p(\frac{x}{|x|})|x|^{m}\log|x|
\end{equation}satisfies 
$$\Delta\Phi_p=\frac{f_p}{\kappa_p}d\mathcal{H}^{d-1}|_{S^{\pm}_\eta} \text{ in $\R^d$ if $m=2k$;}$$
and
$$\Delta\Phi_p=\frac{f_p}{\kappa_p}d\mathcal{H}^{d-1}|_{S^{\pm}_\eta} \text{ in $(\R^d)^{\pm}$, and }  \Phi_p(\cdot,0)=0 \text{ if $m=2k+1$.}$$
We often omit the subscript when there is no ambiguity.

For our argument, it is crucial that $f_p$ has a non-trivial projection into $\mathcal{P}_m$:
\begin{lem}\label{NontrivialProjection}
If $\kappa_p\neq 0$, then there are universal positive constants $c$ and $C$ such that 
$$c\le\|\varphi_p\|_{\mathcal{L}^2(\Sph)}\le C,
$$ and $$\|\Phi_p\|_{C^{0,1}(B_1)} \le C.$$
\end{lem} 
\begin{proof}
Both upper bounds follow from the definitions of $\varphi$, $\Phi$ and Lemma \ref{Lem22}.

For the lower bound for $\varphi$, it suffices to note that we can find $q\in\mathcal{P}_m$ with $\|q\|_{\mathcal{L}^2(\Sph)}=1$ and $q\ge c>0$ along $S^\pm_\eta$ if $\eta$ is small. 
\end{proof} 

\begin{lem}\label{PointwiseApprox}
If $u$ solves \eqref{TOP} in $B_1$, and $p\in\PE\cup\PO$, then $$\|u-\pbar\|_{\mathcal{L}^\infty(B_{1/2})}+\|u-\pbar\|_{H^1(B_{1/2})}\le C(\|u-\pbar\|_{\mathcal{L}^2(B_1)}+\kappa_p)$$ for a universal constant $C$. 

%Moreover, for $A\subset B_{1/2}$, we have 
%$$\|u-\pbar\|_{H^1(A)}\le C \cdot Cap(A)(\|u-\pbar\|_{\mathcal{L}^2(B_1)}+\kappa_p),$$ where $Cap(A)$ is the capacity of the set $A$.
\end{lem} 

\begin{proof}
By Definition \ref{ReplacementOfp}, we have $\pbar\ge0$ along $\HPP$. Thus $\Delta u=0$ inside $\{u-\pbar>0\}$. As a result,  
$$\Delta(u-\pbar)=-\Delta\pbar\ge-fd\mathcal{H}^{d-1}|_{S^{\pm}_\eta} \text{ in $\{u-\pbar>0\}$.}$$
Similarly, inside $\{\pbar-u>0\}$, we have $\Delta\pbar\ge 0.$ Thus 
$$\Delta(u-\pbar)\le\Delta u\le 0 \text{ in $\{\pbar-u>0\}$}.$$
Combining these, we have 
$$\Delta |u-\pbar|\ge-f\HausSurf \text{ in $B_1$.}$$

Meanwhile, let $\zeta$ be a smooth non-negative function on $\Sph$ satisfying 
$$\zeta=0\text{ on $(\Sph)'$ and }\zeta=1 \text{ on $\{|x_d|\ge \frac{\eta}{2}|x'|\}$.}$$ 
Recall the auxiliary function from \eqref{DefPhip}, we have 
$$\Delta(\zeta\Phi)=\frac{1}{\kappa}f\HausSurf+R,$$where the remainder $R$ is universally  bounded in $B_1$. Consequently, we have 
$$\Delta(|u-\pbar|+\kappa\zeta\Phi)\ge-C\kappa \text{ in $B_1.$}$$

Together with Lemma \ref{NontrivialProjection}, this gives the estimates on $\|u-\pbar\|_{\mathcal{L}^\infty(B_{1/2})}$ and $\|u-\pbar\|_{H^1(B_{1/2})}$. 
%The control on $\|u-\pbar\|_{H^1(A)}$ follows by Caccioppoli estimate. 
\end{proof} 

\begin{rem}
Lemma \ref{PointwiseApprox} is the reason why it is preferable to work with the replacement $\pbar$ rather then the original $p$. 
\end{rem}

%%%%%%%%%%%%%%%%%%%%%%%%%%%%%%%%%%%%%%%%%%%%%%%%%%%%
\subsection{Well-approximated solutions}
The heart of this paper is Lemma \ref{Dichotomy}, where we improve the  approximation of a solution $u$  by replacements of polynomials from $\PE$ or $\PO$, defined in \eqref{EvenHomPolyn} and \eqref{OddHomPolyn}.

Let $u$ be a solution to \eqref{TOP} in $B_1$, and let $p\in\mathcal{P}_m$ for $m=2k$ or $2k+1$. The distance between them is denoted by 
\begin{equation}\label{epsup}
\delta(u,p):=\max\{\|u-\overline{p}\|_{H^1(B_1)},\kappa_p\}.
\end{equation} Here we use notations from Definition \ref{ReplacementOfp} and Remark \ref{LongRemarkForReplacement}.

\begin{defi}\label{DefWellApprox}
Given $\eps>0$, we say that $u$ is $\eps$-approximated by $p\in\mathcal{P}_m$ at scale $r>0$, and write 
$$u\in\mathcal{S}_m(p,\eps,r)
$$ if 
$$\delta(u_r,p)<\eps,$$where $u_r(x)=\frac{1}{r^m}u(rx).$
\end{defi} 

We collect some immediate consequences. 

\begin{lem}\label{WeissComparison}
If $u\in\mathcal{S}_{m}(p,\eps,1),$ then $$W_{m}(u;3/4)\le W_m(\pbar;3/4)+C\eps^2
$$for a universal $C$.
\end{lem}

\begin{proof}
With
$\|u-\pbar\|_{H^1(B_1)}\le \eps$,
 we can find $\rho\in[\frac{3}{4},\frac{7}{8}]$ such that 
$$\int_{\partial B_\rho}(u_\nu-\pbar_\nu)^2+(u-\bar p)^2d\mathcal{H}^{d-1}\le C\eps^2.$$ 

A direct computation gives 
\begin{align*}
W_m(\pbar;\rho)-W_m(u;\rho)=&\frac{1}{\rho^{d+2m-2}}\int_{B_\rho}|\nabla(\pbar-u)|^2-2\Delta u(\pbar-u)\\
+\frac{1}{\rho^{d+2m-2}}\int_{\partial B_\rho}2u_\nu&(\pbar-u)-\frac{m}{\rho^{d+2m-1}}\int_{\partial B_\rho}(\pbar-u)^2+2u(\pbar-u).
\end{align*}
With $u\Delta u=0$ and $\pbar\Delta u\le 0$, this implies 
\begin{align*}W_{m}(\pbar;\rho)-W_{m}(u;\rho)&\ge\frac{1}{\rho^{d+2m-2}}\int_{\partial B_\rho}2(\rho u_\nu-mu)(\pbar-u)-m(\pbar-u)^2\\
&\ge-C\eps^2.
\end{align*}
where we have used $\rho u_\nu-mu = \rho (u-\bar p)_\nu - m (u - \bar p)$.

With monotonicity of $W_m$ and homogeneity of $\pbar$, this implies 
$$W_{m}(u;3/4)\le W_m(\pbar;3/4)+C\eps^2.
$$
\end{proof} 

A consequence of Lemma \ref{WeissComparison} is the following relation between $W_m$ and $\eps$.
\begin{lem}\label{WeissComparison2}
If $u\in\mathcal{S}_{m}(p,\eps,1),$ then $$W_{m}(u;3/4)\le C\eps^2, \quad \quad \mbox{if $m$ is odd},$$
and
$$W_{m}(u;3/4)\le C\eps^{1+\frac{2}{d-1}}, \quad \quad \mbox{if $m$ is even,}$$
 with $C$ universal.
\end{lem}

\begin{proof}
By Lemma \ref{WeissComparison} we only need to bound $W_m(\bar p:1)$. 

When $m$ is odd, $W_{2k+1}(\pbar;1)\le 0$ by \eqref{OddPbarEnergy}. 

When $m$ is even, note that 
\begin{equation}\label{EvenPbarEnergy2}
W(\pbar;1)=C\int_{\Sph}|\SphGrad\pbar|^2-\lambda(2k)\pbar^2\le C\eps^{1+\frac{2}{d-1}}
\end{equation} by \eqref{EvenPbarEnergy}. 
  \end{proof} 

\begin{rem}\label{AttractionRepulsion}
The difference between the exponents in Lemma \ref{WeissComparison2} leads to the different rates of convergence in Theorem \ref{MainResult}.
\end{rem}

The following is a version of Lemma B.2 from \cite{FRS}. It follows by quantifying the proof in \cite{FRS} and it is left to the reader:
\begin{lem}\label{PinDown}
Suppose that $u$ is a solution to \eqref{TOP} in $B_1$.

If $u\le p+\eps$ in $B_1$ for some $p\in\PO$ with $\|p\|_{\mathcal{L}^2(\Sph)}\le 1$, then 
$$u=0 \quad \text{in $B_{1-\eps^{1/2}}'\cap\{\ddd p\le -M\eps^{1/2}\}$,}$$where $M$ is a universal constant. 
\end{lem} 

\begin{rem}\label{OneSidedDer}
For a function $w\in C^{1}(B_1\cap\{x_d\ge 0\})$, $\frac{\partial}{\partial x_d}w(x',0)$ denotes the one-sided derivative in the $x_d$-direction taken in $B_1\cap\{x_d\ge0\},$ that is,
$$\frac{\partial}{\partial x_d}w(x',0)=\lim_{t\to 0+}\frac{w(x',t)-w(x',0)}{t}.
$$
\end{rem}

%%%%%%%%%%%%%%%%%%%%%%%%%%%%%%%%%%%%%%%%%%%%%%%%%%%%%%%%%%%%%%%%%%%%%%%%%%%%%%%%%%%%%%%%%%%%%%%%%%%%%%%%
\section{The dichotomy at a point with odd frequency}
Suppose that $u$ is a solution to the thin obstacle problem \eqref{TOP}, and that $0$ is a point with integer frequency. By  results in \cite{GP,FRS}, up to an initial scaling, the solution $u$ is well-approximated in $B_1$ by some homogeneous solution  from either $\PO^+$ or $\PE^+$ as in \eqref{OddHomSolution} and \eqref{EvenHomSolution}.  To get a rate of convergence as in Theorem \ref{MainResult}, we need to improve this approximation at smaller scales. 

This is achieved through the dichotomy as in Lemma \ref{Dichotomy}, which states that at a smaller scale, either the approximation can be improved in a quantified fashion, or the Weiss energy  drops in a quantified fashion. In some sense, this method combines the strengths of the epiperimetric inequality approach  as in \cite{CSV} and the approach by linearization  as in \cite{D}.

In this section and the next, we prove this dichotomy for points with odd and even frequencies, respectively. In the final section of this paper, we show how to deduce the main result from them.  

We state the main lemma for this section:

\begin{lem}[Dichotomy at a point with odd frequency]\label{DichotomyOdd}
Given $k\in\N$, there are universal constants, $\tilde{\eps}$, $r_0$, $c$ small and $C$ big, such that 

If $u\in\mathcal{S}_{2k+1}(p,\eps,1)$ with $\eps<\tilde{\eps}$ and $1\le \|p\|_{\mathcal{L}^2(\Sph)}\le 2$, then we have the following dichotomy: 

a) Either $$W_{2k+1}(u;1)-W_{2k+1}(u;r_0)\ge c\eps^2,$$ and 
$$u\in\mathcal{S}_{2k+1}(p,C\eps,r_0);$$

b) or $$u\in\mathcal{S}_{2k+1}(p',\frac 12\eps,r_0)$$ for some $p'$ with $$\|p'-p\|_{\mathcal{L}^2(\Sph)}\le C\eps.$$
\end{lem} 
Recall the space of well-approximated solutions $\mathcal{S}_{2k+1}$ from Definition \ref{DefWellApprox}, and the Weiss energy from \eqref{Weiss}.

The remaining part of this section is devoted to the proof of Lemma \ref{DichotomyOdd}. We argue by contradiction. 

Suppose, on the contrary, the lemma is not true. Then we find a sequence $(u_n,p_n)$ satisfying 
$$ 1\le\|p_n\|_{\mathcal{L}^2(\Sph)}\le 2, \text{ and } u_n\in\mathcal{S}_{2k+1}(p_n,\eps_n,1) \text{ with $\eps_n\to 0.$}$$
However, neither  a) nor b) holds, that is
\begin{equation}\label{FailingWeiss1}
W_{2k+1}(u_n;1)-W_{2k+1}(u_n;r_0)<\frac{1}{n^2}\eps_n^2,
\end{equation} 
and 
\begin{equation}\label{FailingImprovement}
u_n\notin\mathcal{S}_{2k+1}(p',\frac12\eps_n,r_0) \quad  \forall p'\in\PO \text{ with }\|p'-p_n\|_{\mathcal{L}^2(\Sph)}\le C\eps_n.
\end{equation} The constants $r_0$ and $C_0$ will be chosen depending on universal constants. 

We will choose $r_0\le 1/2$. Hence by monotonicity of the Weiss energy and \eqref{FailingWeiss1}, we have 
\begin{equation}\label{FailingWeiss}
W_{2k+1}(u_n;1)-W_{2k+1}(u_n;1/2)<\frac{1}{n^2}\eps_n^2.
\end{equation}

Corresponding to this sequence $p_n$, we have auxiliary functions $f_n$, $g_n$, $v_n$, $\varphi_n$, $H_n$ and $\Phi_n$, and constants $\kappa_n$ as in Definition \ref{ReplacementOfp}, Remark \ref{LongRemarkForReplacement}, \eqref{Defphip} and \eqref{DefPhip}.

Now with $p_n$ uniformly  bounded in the finite dimensional space $\PO$, we have, up to a subsequence, 
$$p_n\to p_\infty \text{ uniformly in $C^1(B_1\cap\{x_d\ge0\})$ and $C^1(B_1\cap\{x_d\le0\})$.}$$
As a result, we have $p_\infty\in\PO$. Actually, it is in the more restrictive space $\PO^+$ (see \eqref{OddHomSolution}):

\begin{lem}
$\frac{\partial}{\partial x_d} p_\infty\le 0$ in $B_1'$.
\end{lem} 

\begin{proof}
By Lemma \ref{H1Forv}, we have 
\begin{equation}\label{33equation}\|u_n-p_n\|_{H^1(B_1)}\le \|u_n-\overline{p_n}\|_{H^1(B_1)}+\|v_n\|_{H^1(B_1)}\le C\eps_n^{1/2}.
\end{equation}

Suppose $\frac{\partial}{\partial x_d}p_\infty(x',0)=\beta>0$ at some $(x',0)\in B_1'.$ 

Since $\frac{\partial}{\partial x_d}u_n(\cdot,0)\le 0$, we have 
  $\frac{\partial}{\partial x_d}(p_n-u_n)\ge\frac{1}{2}\beta$ in a neighborhood of $(x',0)$ for large $n$. This contradicts \eqref{33equation} eventually. 
\end{proof} 

Since $p_\infty\in\PO$, the following set
\begin{equation*}
\mathcal{N}:=\{\frac{\partial}{\partial x_d} p_\infty=0\}'
\end{equation*}    is of dimension at most $(d-2)$.

\begin{lem}\label{EventualZero}
Given a compact set $K\subset B_1'\backslash\mathcal{N}$,  we can find $N\in\N$ such that $$u_n=0 \text{ on $K$}$$ for all $n\ge N.$
\end{lem} 

\begin{proof}
By compactness of $K$, we can find $\beta>0$ such that $$\ddd p_\infty\le-\beta \text{ on $K$.}$$

With Lemma \ref{PointwiseApprox} and \eqref{33equation}, we have $$u_n\le p_n+C_K\eps_n^{1/2} \text{ in a neighborhood of $K$.}$$ Together with Lemma \ref{PinDown}, this gives the desired result. 
\end{proof} 

Now define the \textit{normalized solutions}
\begin{equation}\label{NormalizedSolution}
\hu_n=\frac{u_n-\pbar_n}{\eps_n}.
\end{equation} With Lemma \ref{PointwiseApprox} and \eqref{DefPhip}, we have that $$\Delta (\hu_n+\frac{\kappa_n}{\eps_n}\Phi_n)=0 \text{ in $B_1^+$}$$ and 
$$  \|\hu_n+\frac{\kappa_n}{\eps_n}\Phi_n\|_{H^1(B_\rho)}\le C(\rho) \text{ for any $\rho<1$.}$$
Thus, up to a subsequence, $\hu_n+\frac{\kappa_n}{\eps_n}\Phi_n$ converges in $L^2_{loc}(B_1)$ to a limit function $h \in H^1_{loc}(B_1),$ which is harmonic in $B_1^+$ and $B_1^-$, and
\begin{equation}\label{VanishingH1}
\|\hu_n+\frac{\kappa_n}{\eps_n}\Phi_n-h\|_{\mathcal{L}^2(B_{7/8})}=o(1) \text{ as $n\to\infty$}.
\end{equation}
Moreover, Lemma \ref{EventualZero} implies that $\hu_n+\frac{\kappa_n}{\eps_n}\Phi_n$ vanishes eventually on any compact subsets of $B_1'\backslash\mathcal{N},$ where $\mathcal{N}$ is a subset of $\HPP$ of dimension at most $(d-2)$.
Consequently, $\hu_n+\frac{\kappa_n}{\eps_n}\Phi_n$ convergences uniformly to $h$ on compact sets in $B_1 \setminus \mathcal N$, which 
 implies that $h=0$ on $B_1'\backslash\mathcal{N}.$
 With $\mathcal{N}$ having $0$ capacity, this gives
 $$\Delta h=0 \text{ in $B_1^+$, and } h=0 \text{ on $B_1'$.}$$ 
% By even symmetry of the functions, the same convergence holds in $B_1\cap\{x_d\le 0\}\backslash\mathcal{N}.$ 

Denote the $(2k+1)$-order Taylor expansion of $h$ at the origin (in $B_1^+$, then evenly reflected to $B_1^-$) by $\sum_{\ell=0}^{2k+1}h_\ell,$ with each $h_\ell$ being the $\ell$-homogeneous part.  Then we have

\begin{lem}\label{InitialImprovement}
There is a universal constant $C$, such that for $r\in(0,1/4)$ we have 
$$\|(\hu_n)_r-h_{2k+1}\|_{\mathcal{L}^2(B_2)}\le Cr(1+|\log r|)+o(1)$$ and 
$$\frac{\kappa_n}{\eps_n}\le Cr+o(1)
 \text{ as $n\to\infty,$}$$where  $(\hu_n)_r(x)=\frac{1}{r^{2k+1}}\hu_n(rx).$
\end{lem}

\begin{proof}
Throughout this proof, for a function $w$, we use $w_r$ to denote its rescaling $$w_r(x)=\frac{1}{r^{2k+1}}w(rx).$$

Firstly, with \eqref{ChangeInRadial} and \eqref{FailingWeiss}, we have
\begin{equation*}
\int_{\partial B_1}|u-u_{\frac{1}{2}}|\le \eps o(1),
\end{equation*} which implies, by maximum principle and the homogeneity of $\pbar$, that 
\begin{equation}\label{35equation}
|\hu-\hu_{\frac 12}|=o(1) \text{ in $B_{7/8}.$}
\end{equation} 

With \eqref{VanishingH1} and regularity of the harmonic function $h$, we have 
\begin{equation}\label{35equation1}
\|\hu+\frac{\kappa}{\eps}\Phi-\sum_{\ell=0}^{2k+1}h_\ell\|_{\mathcal{L}^2(B_{2r})}\le Cr^{2k+2+\frac{d}{2}}+o(1).
\end{equation}
A rescaling gives 
$$\|\hu_{\frac 12}+\frac{\kappa}{\eps}\Phi_{\frac{1}{2}}-\sum_{\ell=0}^{2k+1}(h_\ell)_{\frac 12}\|_{\mathcal{L}^2(B_{4r})}\le Cr^{2k+2+\frac{d}{2}}+o(1).
$$
Combining these with \eqref{35equation}, we get 
$$\|\frac{\kappa}{\eps}(\Phi-\Phi_{\frac{1}{2}})+\sum[(h_\ell)_{\frac{1}{2}}-h_\ell]\|_{\mathcal{L}^2(B_{2r})}\le Cr^{2k+2+\frac{d}{2}}+o(1).
$$
That is, 
$$
\frac{\kappa}{\eps}\frac{\log(2)}{d+4k}\|\varphi|x|^{2k+1}\|_{\mathcal{L}^2(B_{2r})}+\sum_{\ell=0}^{2k}(2^{2k+1-\ell}-1)\|h_\ell\|_{\mathcal{L}^2(B_{2r})}\le Cr^{2k+2+\frac{d}{2}}+o(1),
$$where we used the definition of $\Phi$ from \eqref{DefPhip}, and the orthogonality of $\varphi$ and $h_\ell$ in $\mathcal{L}^2(\Sph)$ for $\ell\le 2k.$

With Lemma \ref{NontrivialProjection}, we can use the bound on the first term to get 
$$\frac{\kappa}{\eps}\le Cr+o(1).
$$
Similarly, the bound on each of the remaining terms gives 
$$\|h_\ell\|_{\mathcal{L}^2(B_{2r})}\le Cr^{2k+2+\frac{d}{2}}+o(1) \text{ for each $0\le\ell\le 2k.$}
$$
Putting these into \eqref{35equation1} gives
\begin{equation}\label{35equation3}
\|\hu_r-h_{2k+1}\|_{\mathcal{L}^2(B_2)}\le Cr(1+|\log r|)+o(1).
\end{equation}This is the desired estimate.
                          \end{proof} 

As an immediate consequence of Lemma \ref{InitialImprovement}, we have 
\begin{equation}\label{AlmostDoneOdd}\|(u_n)_r-(\pbar_n+\eps_n h_{2k+1})\|_{\mathcal{L}^2(B_2)}\le \eps_n[Cr(1+|\log r|)+o(1)].
\end{equation}

\begin{lem}\label{ReplacementOdd}
As $n\to\infty,$ we have
$$\|\overline{p_n+\eps_n h_{2k+1}}-\pbar_n-\eps_n h_{2k+1}\|_{\mathcal{L}^2(B_2)}=\eps_n o(1).$$
\end{lem} 

\begin{proof}In this proof, define $w=\frac{1}{\eps}(\overline{p+\eps h_{2k+1}}-\pbar-\eps h_{2k+1})$.

Firstly, note that by the maximum principle for $\SphLap+\lambda(2k+1)$ in $L^{+}_\eta$, we have 
$|\overline{p+\eps h_{2k+1}}-\pbar|\le C\eps \text{ in $L^{+}$}.$ Thus $|w|\le C \text{ in $L^{+}$.}$

Meanwhile, in any compact subset of $B_1'\backslash\mathcal{N}$,  the same argument as in Lemma \ref{EventualZero} implies $w=0$ for large $n$. We also have $w=0$ along $S^+_\eta$. 

To summarize,  $w$ is a bounded solution to $(\SphLap+\lambda)w=0$ in $L^+_\eta$ with $w=0$ on $S^+_\eta$ and eventually vanishing in any compact subset of $B_1'\backslash\mathcal{N}$, where $\mathcal{N}$ is of dimension at most $(d-2).$

Note that $w$ is even in the $x_d$-direction and vanishes outside $L_\eta$, this implies that $\|w\|_{\mathcal{L}^2(\Sph)}=o(1)$, which gives the desired estimate. \end{proof} 

If we define $p_n'\in\PO$ as  $$p'_n=p_n+\eps_nh_{2k+1},$$ then 
$|p'_n-p_n|\le C\eps_n.$ 

With Lemma \ref{InitialImprovement} and Lemma \ref{ReplacementOdd}, we have 
$$\kappa_{p_n'}\le\kappa_{p_n}+\eps_no(1)\le \eps_n(Cr+o(1)).
$$By choosing $r$ small, we have $\kappa_{p_n'}<\iota\eps_n$ for all large $n$, where $\iota<\frac12$ is a small universal constant to be chosen.

Combining \eqref{AlmostDoneOdd} and Lemma \ref{ReplacementOdd}, we have 
\begin{equation*}\|(u_n)_r-\overline{p'_n}\|_{\mathcal{L}^2(B_2)}\le \eps_n[Cr(1+|\log r|)+o(1)].
\end{equation*}

By choosing $r$ small, depending on universal constants, such that $Cr(1+|\log r|)<\iota/2,$ we have 
$$\|(u_n)_r-\overline{p'_n}\|_{\mathcal{L}^2(B_2)}< \iota\eps_n
$$
for all large $n$. Together with Lemma \ref{PointwiseApprox}, this implies 
$$\|(u_n)_r-\overline{p'_n}\|_{H^1(B_1)}< C\iota\eps_n<\frac{1}{2}\eps_n$$ if $\iota$ is small. 

Consequently, we have $$u_n\in\mathcal{S}_{2k+1}(p_n',\frac12\eps_n,r),$$ contradicting \eqref{FailingImprovement}.

This concludes the proof of Lemma \ref{DichotomyOdd}.

%%%%%%%%%%%%%%%%%%%%%%%%%%%%%%%%%%%%%%%%%%%%%%%%%%%%%%%%%%%%%%%%%%%%%%%%%%%%%%%%%%%%%%%%%%%%%%%%%%%%%%%%
\section{The dichotomy at a point with even frequency}
In this section, we establish a dichotomy similar to Lemma \ref{DichotomyOdd} but at a contact point with even frequency. We also explain how to get  a $\log$-epiperimetric inequality with a slightly improved exponent then the one in  Colombo-Spolaor-Velichkov \cite{CSV}.

The ideas are similar to those in the previous section. We only sketch the proof. 

This main lemma for this section is:

\begin{lem}[Dichotomy at a point with even frequency]\label{DichotomyEven}
Given $k\in\N$, there are universal constants, $\tilde{\eps}$, $r_0$, $c$ small and $C$ big, such that 

If $u\in\mathcal{S}_{2k}(p,\eps,1)$ with $\eps<\tilde{\eps}$ and $1\le \|p\|_{\mathcal{L}^2(\Sph)}\le 2$, then we have the following dichotomy: 

a) Either $$W_{2k}(u;1)-W_{2k}(u;r_0)\ge c\eps^2$$ and 
$$u\in\mathcal{S}_{2k}(p,C\eps,r_0);$$

b) or $$u\in\mathcal{S}_{2k}(p',\frac12\eps,r_0)$$ for some $p'$ with $$\|p'-p\|_{\mathcal{L}^2(\Sph)}\le C\eps.$$
\end{lem} 

We prove this lemma by contradiction. 

Suppose the lemma is not true, then we find a sequence $(u_n,p_n)$ satisfying 
$$ 1\le\|p_n\|_{\mathcal{L}^2(\Sph)}\le 2, \text{ and } u_n\in\mathcal{S}_{2k}(p_n,\eps_n,1) \text{ with $\eps_n\to 0.$}$$
However, 
\begin{equation}\label{FailingWeissEven}
W_{2k}(u_n;1)-W_{2k}(u_n;\frac{1}{2})<\frac{1}{n^2}\eps_n^2,
\end{equation} 
and 
\begin{equation}\label{FailingImprovementEven}
u_n\notin\mathcal{S}_{2k}(p',\frac12\eps_n,r_0) \quad  \forall p'\in\PE \text{ with }\|p'-p_n\|_{\mathcal{L}^2(\Sph)}\le C\eps_n.
\end{equation} The constants $r_0$ and $C$ will be chosen depending on universal constants. 

Similar to the previous case, up to a subsequence, we have 
$$p_n\to p_\infty\in\PE \text{ in $C^\infty(B_1)$.}$$
With Lemma \ref{BoundsForKappa}, we have $p_n\ge-v_n\ge-C\eps_n^{\frac{2}{d-1}}$ on $B_1'$. Thus 
$$p_\infty\ge0 \text{ on $B_1'$.}$$

The set where $p_\infty$ vanishes on $B_1'$,  $\mathcal{N}=\{p_\infty=0\}'$, has dimension at most $(d-2)$. 

Define normalized solutions $\hu_n$ as in \eqref{NormalizedSolution}, we have, up to a subsequence,
\begin{equation*}
\|\hu_n+\frac{\kappa_n}{\eps_n}\Phi_n-h\|_{\mathcal{L}^2(B_{7/8})}=o(1) \text{ as $n\to\infty$}
\end{equation*} for some $h$ satisfying $$\Delta h=0 \text{ in $B_1$.}$$

With similar ideas as in Lemma \ref{InitialImprovement}, we can rule out lower order terms in the Taylor polynomial of $h$ at $0$ and obtain for $r\in(0,1/4)$,
\begin{equation*}\|(u_n)_r-(\pbar_n+\eps_n h_{2k})\|_{\mathcal{L}^2(B_2)}\le \eps_n[Cr(1+|\log r|)+o(1)].
\end{equation*}

If we choose $r$ small, then $p'_n=p_n+\eps_nh_{2k}\in\PE$ satisfies $|p'_n-p_n|\le C\eps_n,$ and
$$\|(u_n)_r-\overline{p'_n}\|_{\mathcal{L}^2(B_2)}< \iota\eps_n
\text { and }\kappa_{p_n'}\le\iota\eps_n
$$for large $n$, where $\iota$ is a universally small constant. 

An application of Lemma \ref{PointwiseApprox} again gives $$u_n\in\mathcal{S}_{2k}(p_n',\frac{1}{2}\eps_n,r)$$ if $\iota$ is small, which contradicts \eqref{FailingImprovementEven}. 

This completes the proof for Lemma \ref{DichotomyEven}.

\begin{rem}\label{ImprovedEpi}
We sketch how similar ideas lead to a $\log$-epiperimetric inequality for the $2k$-Weiss energy functional. We get an improved exponent than the one currently known in the literature. 

To be precise, let $w$ be a $2k$-homogeneous function satisfying $$w\ge 0 \text{ on $\HPP$,}$$ and $$\|w\|_{\mathcal{L}^2(\Sph)}\le 1, \quad |W_{2k}(w;1)|\le 1,$$ then we will show
\begin{equation}\label{EpiIneq}
W_{2k}(w;1)-W_{2k}(u;1)\ge cW_{2k}(w;1)^{1+\frac{d-3}{d+1}},
\end{equation} where $u$ is the solution to \eqref{TOP} with $u|_{\Sph}=w$, and $c$ is a universal constant. 

A similar result is known in \cite{CSV} with the exponent on the right-hand side as $1+\frac{d-2}{d}.$

It suffices to prove \eqref{EpiIneq} under the assumption $\|w\|_{\mathcal{L}^2(\Sph)}=1.$ For such $w$, choose \textit{$p\in\PE$ that minimizes $\delta(w,p)$ from \eqref{epsup}}. 

With $W(p;1)=0$, $\Delta p=0$ and the homogeneity of $p$, we have 
$$W(w;1)\le\int_{B_1}|\nabla w-\nabla p|^2\le C\delta(w,p)^2+C\|\pbar-p\|^2_{H^1(B_1)}.$$With homogeneity and harmonicity of $p$, we also have 
$$W(\pbar;1)=W(\pbar-p;1)=C[\int_{L_{\eta}}|\nabla_{\Sph}(\pbar-p)|^2-\lambda(2k)(\pbar-p)^2],$$
where the definitions of $L_\eta$  and $\lambda(2k)$ are given at the beginning of Subsection \ref{TBLP}. By making $\eta$ smaller, if necessary, we have $\int_{L_{\eta}}|\nabla_{\Sph}(\pbar-p)|^2-\lambda(2k)(\pbar-p)^2\sim\|\pbar-p\|^2_{H^1(B_1)}$. Therefore,  
$$W(w;1)\le C\delta(w,p)^{1+\frac{2}{d-1}},$$ where we used \eqref{EvenPbarEnergy2}.

With $u=w$ on $\Sph$ and $w\ge 0 \text{ on $\HPP$,}$ we have 
$$ W(w;1)-W(u;1)\ge \int_{B_1}|\nabla(w-u)|^2.$$

Therefore, it suffices to show that for $\delta(w,p)$ small, we have 
\begin{equation}\label{RemLower}\int_{B_1}|\nabla(w-u)|^2\ge c\delta(w,p)^2
\end{equation} for some universal constant $c$. 

Suppose, on the contrary, this fails. Then we find a sequence $(w_n,u_n,p_n)$ as described above with $\delta_n=\delta(w_n,p_n)\to 0$ but 
\begin{equation}\label{RemLower2}\int_{B_1}|\nabla(w_n-u_n)|^2\le \frac{1}{n^2}\delta_n^2.\end{equation}
Similar ideas as in the proof for Lemma \ref{DichotomyEven} gives, for large $n$,  $$\delta(u_n,p_n')<\frac{1}{2}\delta_n$$ for some $p_n'\in\PE$, where \eqref{RemLower2} can be used in place of \eqref{FailingWeissEven} to control terms with lower homogeneities. With \eqref{RemLower2}, this gives $\delta(w_n,p_n')<\delta_n(w_n,p_n),$ contradicting the minimizing property of $p_n.$
\end{rem}

%%%%%%%%%%%%%%%%%%%%%%%%%%%%%%%%%%%%%%%%%%%%%%%%%%%%%%%%%%%%%%%%%%%%%%%%%%%%%%%%%%%%%%%%%%%%%%%%%%%%%%%%
\section{Convergence rate to the blow-up profile}
In this final section, we prove our main result Theorem \ref{MainResult}. Our result on stratification of contact points with integer frequencies, Theorem \ref{MainStratification}, follows with Whitney's extension theorem and the implicit function theorem. See, for instance, \cite{GP}.

We first give a technical lemma about sequences. We will apply this to the sequences of Weiss energy and errors in approximations at different scales.
\begin{lem}\label{Sequences}
Let $(w_n)$ and $(e_n)$ be two sequences of nonnegative real numbers with $e_0 \le 1$. Suppose that for some constants, $A$ big, $a$ small and $\gamma\in(0,1]$, we have $$w_{n+1}\le A \, e_n^{1+\gamma},$$ and  the following dichotomy:
\begin{itemize}
\item{ either $w_{n+1}\le w_n-ae_n^2$ and $e_{n+1}=Ae_n$; }
\item{ or $w_{n+1}\le w_n$ and $e_{n+1}=\frac{1}{2}e_{n}$}.
\end{itemize}
Then  
$$\sum e_n<\sigma(e_0) \to 0 \quad \mbox{as $e_0 \to 0$},$$
and \begin{equation}\label{OddSum}\sum_{n\ge N}e_n\le C(1-c)^N \text{ if $\gamma=1$;}\end{equation} and 
\begin{equation}\label{EvenSum}\sum_{n\ge N}e_n \le CN^{\frac{-\gamma}{1-\gamma}}\text{ if $\gamma\in(0,1)$.}\end{equation}
Here $c\in(0,1)$ and $C$ are constants depending only on $A$, $a$ and $\gamma.$
\end{lem} 

\begin{proof}
Define a new sequence $$\alpha_n:=w_n+\mu e_n^2.$$ If $\mu>0$ is small enough, then we have 
\begin{equation}\label{alphaDecay}\alpha_{n+1}\le \alpha_n-c \alpha_n^{\frac{2}{1+\gamma}},
\end{equation}
and $\alpha_1 \le C e_0^{1+\gamma}$.

For $\gamma=1$, we have $\alpha_n\le (1-c)^{n-1}\alpha_1,$ which gives the desired estimate for this case. 

For $\gamma\in(0,1)$, from \eqref{alphaDecay} we have
that for all $n\ge 1$,
 $$\alpha_n\le C(n+M)^{-\frac{1+\gamma}{1-\gamma}},$$
 with $M \to \infty$ as $e_0 \to 0$, and the first estimate follows.
Meanwhile, by our definition of $\alpha_n$, we have $e_n^2\le C(\alpha_n-\alpha_{n+1}).$
Thus
\begin{align*}
\sum_{n=N}^{2N}e_n&\le C\sum_{N}^{2N}(\alpha_n-\alpha_{n+1})^{1/2}\\
&\le CN^{1/2} \left [ \sum_{N}^{2N}(\alpha_n-\alpha_{n+1})\right]^{1/2}\\
&\le CN^{1/2} \alpha_{N}^{1/2}\\
&\le CN^{-\frac{\gamma}{1-\gamma}}.
\end{align*}This gives the desired estimate for $\gamma\in(0,1)$.
\end{proof} 

Now we give the proof of our main result.
\begin{proof}[Proof of Theorem \ref{MainResult}]
Suppose that $0\in\Lambda_m(u)$ for some $m\in\N$, then up to an initial rescaling, we have 
$$u\in\mathcal{S}_m(p,\eps,1), \quad \quad \|p\|_{L^2(B_1)}=3/2,$$ for some $\eps<\tilde{\eps}$. Here $\tilde{\eps}$ is the constant from Lemma \ref{Dichotomy}, and the solution class $\mathcal{S}_m$ is from Definition \ref{DefWellApprox}.

As the initial set up, let $p_0=p$, $\rho_0=1$, $e_0=\eps$, and $w_0=W_m(u;1)$. 

Suppose that we have found $p_n$, $e_n$ $w_n$ small, and $\rho_n\in(0,1)$ such that $u\in\mathcal{S}_{m}(p_n,e_n,\rho_n)$ with $e_n<\tilde{\eps}$. Then we apply Lemma \ref{Dichotomy}. If possibility a) happens in the dichotomy, we let $p_{n+1}=p_n$, $e_{n+1}=Ce_n$. If possibility b) happens, we let $p_{n+1}=p'$ and $e_{n+1}=\frac{1}{2}e_n.$ In both cases, we let $\rho_{n+1}=r_0\rho_n$ and $w_{n+1}=W_m(u;\rho_{n+1}).$

By Lemma \ref{Dichotomy} and Lemma \ref{WeissComparison2},  the sequences $(w_n)$ and $(e_n)$ satisfy the assumptions in Lemma \ref{Sequences}, with $\gamma=1$ if $m$ is odd, and $\gamma=\frac{2}{d-1}$ if $m$ is even. In particular, we have  $\sum e_n<\tilde{\eps}$ along the sequence if $e_0$ is chosen small enough. Consequently, Lemma \ref{Dichotomy} can be applied indefinitely. 

Now note that $\|p_{n+1}-p_n\|_{\mathcal{L}^2(B_1)}\le Ce_n.$ The summability of $(e_n)$ implies the convergence of $p_n$ to some limit $p_\infty$. 

If we denote by $u_r$ the rescaled solution $u_r(x)=\frac{1}{r^m}u(rx).$ When $m$ is odd, we use \eqref{OddSum} to get $\|u_{r_0^n}-p_\infty\|_{H^1(B_1)}\le C(1-c)^n.$ This gives the estimate in \eqref{MainResultOdd}. When $m$ is even, we use \eqref{EvenSum} to get $\|u_{r_0^{n}}-p_\infty\|_{H^1(B_1)}\le Cn^{-\frac{2}{d-3}}.$ This gives the estimate in \eqref{MainResultEven}.
\end{proof}

 The stratification in Theorem \ref{MainStratification} follows by the same strategy as in Garofalo-Petrosyan \cite{GP} or Colombo-Spolaor-Velichkov \cite{CSV}. In the following remark, we point out that $\Lambda_1(u)$ and $\Lambda_3^0(u)$ always lie in the interior of the contact set. 

\begin{rem}\label{Interior}
Suppose $0\in\Lambda_{2k+1}(u)$, then there is $p\in\PO^+$ such that $$u(x)=p(x)+O(|x|^{2k+1+\alpha})$$ as $x\to0.$ 

If $0\in\Lambda_1(u)$, then $p$ is a positive multiple of $-|x_d|$. 

If $0\in\Lambda_3(u)$,  we have $$p(x',x_d)=-|x_d|(p_1(x')+x_d^2p_2(x',x_d)),$$ where $p_1$ is a $2$-homogeneous polynomial with $p_1\ge0$ on $\HPP.$ The zero stratum of $\Lambda_3(u)$ is defined as those points where $p_1$ depends on all $(d-1)$-variables. This implies $p_1>0$ on $(\Sph)'.$

Consequently, if $0\in\Lambda_1(u)$ or $\Lambda_3^0(u)$, then $\ddd p<0$ on $(\Sph)'$. With Lemma \ref{PinDown}, we have $u=0$ in $B_r'$ for some $r>0.$
\end{rem}

%%%%%%%%%%%%%%%%%%%%%%%%%%%%%%%%%%%%%%%%%%%%%%%%%%%

\end{document}